\documentclass{amsart}
\usepackage[foot]{amsaddr}
\usepackage{amssymb,amsmath,amsfonts,mathptmx,cite,mathrsfs,url}
\usepackage{eucal}

\theoremstyle{plain}
\newtheorem{theorem}{Theorem}
\newtheorem{proposition}{Proposition}
\newtheorem{lemma}{Lemma}
\newtheorem{claim}{Claim}
\newtheorem{corollary}{Corollary}
\newtheorem{question}{Question}

\theoremstyle{remark}
\newtheorem*{remark}{Remark}

\DeclareSymbolFont{rsfscript}{OMS}{rsfs}{m}{n}
\DeclareSymbolFontAlphabet{\mathrsfs}{rsfscript}

\def\fb{finitely based}
\def\nfb{non\-finitely based}

\newcommand{\mA}{\mathcal{A}}
\newcommand{\mB}{\mathcal{B}}
\newcommand{\mC}{\mathcal{C}}

\newcommand{\mP}{\mathcal{P}}
\newcommand{\mR}{\mathcal{R}}
\newcommand{\mS}{\mathcal{S}}
\newcommand{\mT}{\mathcal{T}}

\newcommand{\ais}{ai-semi\-ring}

\DeclareMathOperator{\var}{var}
\DeclareMathOperator{\End}{End}

\usepackage{tikz}
\usetikzlibrary{arrows}
\usetikzlibrary{arrows,shapes}
\usetikzlibrary{tqft}
\usetikzlibrary{calc}

\makeatletter

\renewcommand*\subjclass[2][2010]{\def\@subjclass{#2}\@ifundefined{subjclassname@#1}{\ClassWarning{\@classname}{Unknown edition (#1) of Mathematics Subject Classification; using '2010'.}}{\@xp\let\@xp\subjclassname\csname subjclassname@#1\endcsname}}

\renewcommand{\subjclassname}{\textup{2010} Mathematics Subject Classification}

\makeatother

\begin{document}

\title[The finite basis problem for endomorphism semirings]{The finite basis problem for the endomorphism semirings\\ of finite semilattices}
\thanks{Igor Dolinka was supported by the Personal grant F-121 of the Serbian Academy of Sciences and Arts.\linebreak Sergey~V. Gusev was supported  by the Ministry of Science and Higher Education of the Russian Federation (project FEUZ-2023-0022)}

\author[I. Dolinka]{Igor Dolinka}
\address{{\normalfont (I. Dolinka) Department of Mathematics and Informatics, University of Novi Sad, 21101 Novi Sad, Serbia}}
\email{dockie@dmi.uns.ac.rs}

\author[S. V. Gusev]{Sergey V. Gusev}
\address{{\normalfont (S. V. Gusev) Institute of Natural Sciences and Mathematics, Ural Federal University, 620000 Ekaterinburg, Russia}}
\email{sergey.gusb@gmail.com}
\email{m.v.volkov@urfu.ru}

\author[M. V. Volkov]{Mikhail V. Volkov}
\address{{\normalfont (M. V. Volkov) 620075 Ekaterinburg, Russia}}

\begin{abstract}
For every semilattice $\mA=(A,+)$, the set $\End(\mA)$ of its endomorphisms forms a semiring under pointwise addition and composition. We prove that if $\mA$ is finite, then the endomorphism semiring  $\End(\mA)$ has a finite identity basis if and only if $|A|\le 2$.
\end{abstract}



\maketitle

\section*{Introduction}
\label{sec:introduction}

\subsection*{Semirings} A \emph{semiring} is an algebra $(S,+,\cdot)$ with two binary operations, addition $+$ and multiplication $\cdot$, such that
\begin{itemize}
\item[(1)] $(S,+)$ is a commutative semigroup, that is,
\[
a+(b+c)=(a+b)+c\ \text{ and }\  a+b=b+a\  \text{ for all }\ a,b,c\in S;
\]
\item[(2)] $(S,\cdot)$ is a semigroup, that is,
\[
a(bc)=(ab)c\ \text{ for all }\ a,b,c\in S;
\]
\item[(3)] the left and right hand distributive laws hold, that is,
\[
a(b+c)=ab+ac\ \text{ and }\ (a+b)c=ac+bc\ \text{ for all }\ a,b,c\in S.
\]
\end{itemize}

Quoting from Golan's monograph \cite{Golan}, ``Semirings abound in the mathematical world around us. Indeed, the first mathematical structure we encounter---the set of natural num\-bers---is a semiring.'' In this paper, the following rich source of semirings plays a role: for an arbitrary commutative semigroup $(A,+)$, the set of all its \emph{endomorphisms}, i.e., self-maps $A\to A$ respecting $+$, forms a semiring under pointwise addition and composition.

\subsection*{Additively idempotent semirings} As the theory of semirings has developed, it has become clear that many questions about semirings eventually reduce to two special cases: when the additive semigroup is an abelian group and when it is a \emph{semilattice}, that is, a commutative and idempotent semigroup. Semirings whose additive semigroups are abelian groups are nothing but rings. A semiring whose additive semigroup is a semilattice is referred to as an \emph{additively idempotent semiring} (\ais, for short). AI-semirings naturally arise in various areas of mathematics (e.g., idempotent analysis, tropical geometry, and optimization) and computer science, where rings and other ``more classical'' algebras fail to provide adequate tools; see the contributions in the conference volumes \cite{Guna98,LiMa:2005} for diverse examples. The reader should be warned, however, that besides the name, there are no immediate connections between \ais{}s and Artificial Intelligence.

\subsection*{``Universal'' \ais{}s} In ring theory, endomorphism rings of abelian groups serve as ``universal'' rings, as every ring embeds into the endomorphism ring of a suitable abelian group. Endomorphism semirings of semilattices play the same universal role for \ais{}s: every \ais{} embeds into the endomorphism semiring of a suitable semilattice (\!\cite[Theorem 2.4]{KP05}, preceded by \cite[Theorem 2.2(1)]{KRM97} in the case of finite \ais{}s). Refer to \cite{JK09,JKM09} for a study of structure properties of endomorphism semirings of semilattices, which has revealed further parallels with endomorphism rings of abelian groups (e.g., regarding properties related to congruences such as subdirect irreducibility and simplicity).

\subsection*{Our contribution} Our main result highlights a feature of endomorphism semirings of semilattices that stands in sharp contrast to properties known for rings. Recall that every finite ring (in particular, every endomorphism ring of a finite abelian group) is known to possess a finite identity basis (\!\cite{Kr,Lv}; see also \cite[Section 4.5]{Sapir} for a modern, streamlined proof). Here, we show that the endomorphism semiring of a finite semilattice admits no finite identity basis unless the semilattice has at most two elements. This result thus completes the investigation of the finite basis problem for the endomorphism semirings of finite semilattices, initiated by the first-named author more than 15 years ago~\cite{Dolinka09b}. It also extends (and improves) the solution to the finite basis problem for the endomorphism semirings of finite chains obtained by the second- and third-named authors~\cite{GusVol25}; we note that no results from~\cite{GusVol25} are used here, except for a few auxiliary facts.

\subsection*{Structure of the paper} In Section~\ref{sec:certain}, we discuss semilattices and their endomorphism semirings. We also introduce three six-element \ais{}s that play a crucial role in several of our arguments and provide the necessary background on semiring identities and varieties. Section~\ref{sec:methods} presents an overview of the three methods developed to address various instances of the finite basis problem for \ais{}s. Handling the endomorphism semirings of finite semilattices requires a combination of all these methods; this is the focus of Section~\ref{sec:end}. We conclude with a discussion of possible future work in Section~\ref{sec:conclude}.

\section{Preliminaries}
\label{sec:certain}
\subsection{Prerequisites} A fair effort has been made to keep the paper reasonably self-con\-tained. We use only a few standard notions from group theory and universal algebra, along with the concept of presenting semigroups by generators and relations. All of these can be found in the early chapters of the canonical textbooks \cite{BuSa81,Hall:1959,Howie:1995}. No specific knowledge of semirings or semigroups is assumed.

\subsection{Semilattices as ordered sets} We have defined a semilattice as a commutative idempotent semigroup. It is well known that this algebraic definition is equivalent to the following order-theoretic one: an upper semilattice is a partially ordered set in which every pair of elements has a least upper bound. Indeed, if a semigroup $(A,+)$ is commutative and idempotent, then the relation $\leqslant$ on $A$ defined by
\begin{equation}\label{eq:semilattice order}
a\leqslant a'\iff a+a'=a'\ \text{ for }\ a,a'\in A
\end{equation}
is a partial order, and for all $a,b\in A$, their sum $a+b$ is the least upper bound of $a$ and $b$ with respect to the order~\eqref{eq:semilattice order}. Conversely, if a partially ordered set $(A,\leqslant)$ contains a least upper bound for every pair $(a,b)\in A\times A$, then defining $a+b$ as the least upper bound of $a$ and $b$ yields an associative, commutative, and idempotent operation on $A$. Moreover, the order of $(A,\leqslant)$ can be recovered from this operation via~\eqref{eq:semilattice order}.

While the algebraic and order-theoretic viewpoints on semilattices are fully equivalent, in what follows we represent semilattices by Hasse diagrams rather than Cayley tables, as diagrams encode the addition operation both more concisely and more visually.

If $(S,+)$ is the additive semilattice of an \ais{} $(S,+,\cdot)$, then the distributive laws imply that multiplication is \emph{compatible} with the order~\eqref{eq:semilattice order}; that is, for all $a,b,c\in S$,
\[
a\leqslant b \implies ac\leqslant bc \ \text{ and  } \ ca\leqslant cb.
\]
This means that $(S,\leqslant,\cdot)$ is an \emph{ordered semigroup}. This connection explains why \ais{}s are sometimes called \emph{semilattice-ordered semigroups} in the literature.

\subsection{Endomorphism semirings of semilattices} By an \emph{endomorphism} of a semilattice $(A,+)$, we mean any self-map $\alpha\colon A\to A$ that preserves $+$ in the sense that, for all $a,b\in A$,
\[
(a+b)\alpha=a\alpha+b\alpha.
\]
Every semilattice endomorphism  $\alpha\colon A\to A$  preserves the order~\eqref{eq:semilattice order}, that is, for all $a,b\in A$,
\[
a\leqslant b\implies a\alpha\leqslant b\alpha.
\]

Denote set of all endomorphisms of $\mA=(A,+)$ by $\End(\mA)$. We equip $\End(\mA)$ with addition and multiplication as follows. For all endomorphisms $\alpha,\beta\colon A\to A$, their sum $\alpha+\beta$ is defined pointwise by letting for each $a\in A$,
\[
a(\alpha+\beta):=a\alpha+a\beta.
\]
The product $\alpha\beta$ is just the composition of $\alpha$ and $\beta$ as maps; that is, for each $a\in A$,
\[
a(\alpha\beta):=(a\alpha)\beta.
\]
It is easy to verify that $(\End(\mA),+,\cdot)$ is an \ais{} called the \emph{endomorphism semiring} of $\mA$. It is this \ais{} that, under the assumption that the semilattice $\mA$ is finite, is the main object of this paper.

We note that the pointwise definition of addition on $\End(\mA)$ implies that the order~\eqref{eq:semilattice order} on $\End(\mA)$ is also of a pointwise nature. That is, for all endomorphisms $\alpha,\beta\colon A\to A$,
\[
\alpha\leqslant\beta\iff a\alpha\leqslant a\beta \ \text{ for all }\ a\in A.
\]

Sometimes, to lighten notation, we take the liberty of denoting the endomorphism semiring of a semilattice $\mA$ simply by $\End(\mA)$ rather than by $(\End(\mA),+,\cdot)$.

\subsection{Three six-element \ais{}s}
\label{subsec:b1a1}

The two six-element monoids
\begin{align*}
A^1_2&:=\langle e,a\mid eae=e^2=e,\ aea=a,\ a^2=0\rangle=\{1,e,a,ae,ea,0\},\\
B_2^1&:=\langle c,d\mid cdc=c,\ dcd=d,\ c^2=d^2=0\rangle=\{1,c,d,cd,dc,0\}
\end{align*}
are presented here both by generators and relations (in the class of monoids with zero) and by explicit lists of their elements. These monoids play a distinguished role in both the structure theory of semigroups and the theory of semigroup varieties. For the present paper, it is important that each of them admits addition making it an \ais{}.

The monoid $B_2^1$ (commonly called the \emph{six-element Brandt monoid}) is known to admit a unique semilattice addition $+$ over which its multiplication distributes. Figure~\ref{fig:b21} shows the Hasse diagram of the partially ordered set $(B_2^1,\leqslant)$, where $\leqslant$ is the unique upper semilattice order compatible with multiplication in $B_2^1$. We let $\mB_2^1$ stand for the \ais{} obtained by equipping the monoid $B_2^1$ with the corresponding addition.
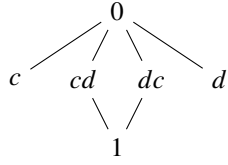
\begin{figure}[htb]
\begin{tikzpicture}
[scale=0.9]
\node at (0,-1)    (U)  {$1$};
\node at (-0.5,0) (EA)  {$cd$};
\node at (-1.5,0) (E) {$c$};
\node at (0.5,0)  (AE) {$dc$};
\node at (1.5,0)  (A)  {$d$};
\node at (0,1)    (Z)  {$0$};
\draw (A)  -- (Z);
\draw (EA) -- (Z);
\draw (AE) -- (Z);
\draw (E)  -- (Z);
\draw (EA) -- (U);
\draw (AE) -- (U);
\end{tikzpicture}
\caption{The only compatible upper semilattice order on $B_2^1$}\label{fig:b21}
\end{figure}

The six-element monoid $A_2^1$ admits two different \ais{} structures. One of the two possible upper semilattice orders compatible with its multiplication is shown in Fig.~\ref{fig:a21}.
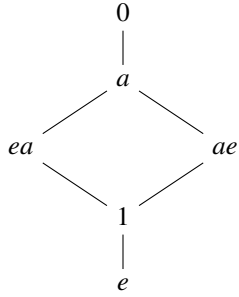
\begin{figure}[h]
\begin{tikzpicture}
[scale=0.9]
\node at (0,-1)   (E)  {$e$};
\node at (0,0)    (U)  {$1$};
\node at (-1.5,1) (EA) {$ea$};
\node at (1.5,1)  (AE) {$ae$};
\node at (0,2)    (A)  {$a$};
\node at (0,3)    (Z)  {$0$};
\draw (A)   -- (Z);
\draw (EA)  -- (A);
\draw (AE)  -- (A);
\draw (EA)  -- (U);
\draw (AE)  -- (U);
\draw (E)   -- (U);
\end{tikzpicture}
\caption{A compatible upper semilattice order on $A_2^1$}\label{fig:a21}
\end{figure}

The order in Fig.~\ref{fig:a21} comes from a well-known faithful representation of the monoid $A_2^1$ as the monoid of all endomorphisms of the chain $0<1<2$ that fix the top element of the chain. The representation is shown in Fig.~\ref{fig:A2asOPT}.
\begin{figure}[b]
\centering
\begin{tikzpicture}
[scale=0.78]
\foreach \x in {0.75,2.25,3.5,5,6.25,7.75,9,10.5,11.75,13.25,14.5,16} \foreach \y in {0.5,1.5,2.5} \filldraw (\x,\y) circle (2pt);
	\draw (0.75,0.5) edge[-latex] (2.25,0.5);
	\draw (0.75,1.5) edge[-latex] (2.25,1.5);
    \draw (0.75,2.5) edge[-latex] (2.25,2.5);
	\draw (3.5,0.5) edge[-latex] (5,1.5);
	\draw (3.5,1.5) edge[-latex] (5,1.5);
    \draw (3.5,2.5) edge[-latex] (5,2.5);
	\draw (6.25,0.5) edge[-latex] (7.75,0.5);
	\draw (6.25,1.5) edge[-latex] (7.75,2.5);
    \draw (6.25,2.5) edge[-latex] (7.75,2.5);
	\draw (9,0.5) edge[-latex] (10.5,1.5);
	\draw (9,1.5) edge[-latex] (10.5,2.5);
    \draw (9,2.5) edge[-latex] (10.5,2.5);
	\draw (11.75,0.5) edge[-latex] (13.25,0.5);
	\draw (11.75,1.5) edge[-latex] (13.25,0.5);
    \draw (11.75,2.5) edge[-latex] (13.25,2.5);
    \draw (14.5,0.5) edge[-latex] (16,2.5);
	\draw (14.5,1.5) edge[-latex] (16,2.5);
    \draw (14.5,2.5) edge[-latex] (16,2.5);
\foreach \x in {0.5,2.5,3.25,5.25,6,8,8.75,10.75,11.5,13.5,14.25,16.25} \node at (\x,0.5) {\tiny 0};
\foreach \x in {0.5,2.5,3.25,5.25,6,8,8.75,10.75,11.5,13.5,14.25,16.25} \node at (\x,1.5) {\tiny 1};
\foreach \x in {0.5,2.5,3.25,5.25,6,8,8.75,10.75,11.5,13.5,14.25,16.25} \node at (\x,2.5) {\tiny 2};
\node at (1.5,-0.15)   {$1$};
\node at (4.25,-0.15)  {$ea$};
\node at (7,-0.15)     {$ae$};
\node at (9.75,-0.15)  {$a$};
\node at (12.5,-0.15)  {$e$};
\node at (15.25,-0.15) {$0$};
\end{tikzpicture}
\caption{Representing the elements of ${A}^1_2$ by the endomorphisms of~the chain $0<1<2$ that fix 2}\label{fig:A2asOPT}
\end{figure}
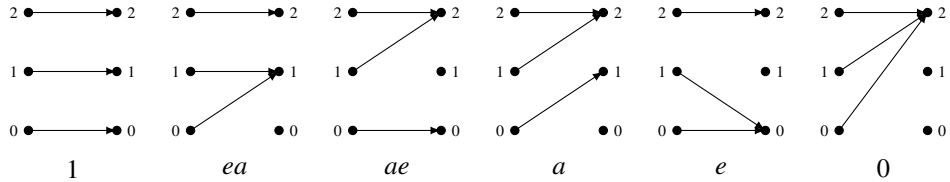
One readily sees that the order on $A_2^1$ shown in Fig.~\ref{fig:a21} corresponds to the pointwise order of endomorphisms.

We denote by $\mA_2^1$ the \ais{} obtained by equipping the monoid $A_2^1$ with addition induced by the order in Fig.~\ref{fig:a21}.

It turns out that the order on $A_2^1$ dual to that shown in Fig.~\ref{fig:a21} is also compatible with multiplication.
\begin{figure}[ht]
\begin{tikzpicture}
[scale=0.9]
\node at (0,-1)   (Z)  {$0$};
\node at (0,0)    (A)  {$a$};
\node at (-1.5,1) (EA) {$ea$};
\node at (1.5,1)  (AE) {$ae$};
\node at (0,2)    (U)  {$1$};
\node at (0,3)    (E)  {$e$};
\draw (A)   -- (Z);
\draw (EA)  -- (A);
\draw (AE)  -- (A);
\draw (EA)  -- (U);
\draw (AE)  -- (U);
\draw (E)   -- (U);
\end{tikzpicture}
\caption{Another compatible upper semilattice order on $A_2^1$}\label{fig:a21-2}
\end{figure}
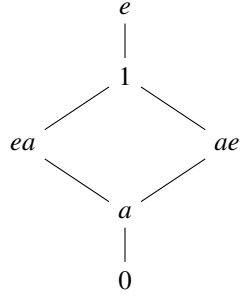
This dual order (shown in Fig.~\ref{fig:a21-2}) also comes from a faithful representation of the monoid $A_2^1$ by endomorphisms of the chain $0<1<2$ but this time the image of the representation consists of the endomorphisms  that fix the bottom element of the chain. The representation is shown in Fig.~\ref{fig:A2asOPT-2}.
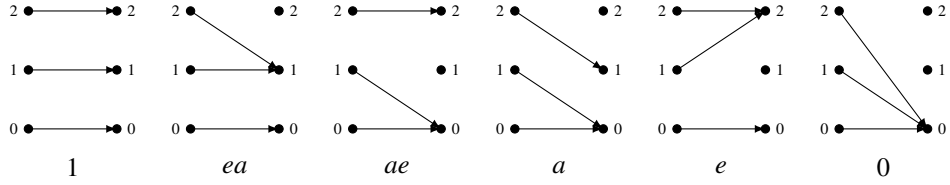
\begin{figure}[htb]
\centering
\begin{tikzpicture}
[scale=0.78]
\foreach \x in {0.75,2.25,3.5,5,6.25,7.75,9,10.5,11.75,13.25,14.5,16} \foreach \y in {0.5,1.5,2.5} \filldraw (\x,\y) circle (2pt);
	\draw (0.75,0.5) edge[-latex] (2.25,0.5);
	\draw (0.75,1.5) edge[-latex] (2.25,1.5);
    \draw (0.75,2.5) edge[-latex] (2.25,2.5);
	\draw (3.5,0.5) edge[-latex] (5,0.5);
	\draw (3.5,1.5) edge[-latex] (5,1.5);
    \draw (3.5,2.5) edge[-latex] (5,1.5);
	\draw (6.25,0.5) edge[-latex] (7.75,0.5);
	\draw (6.25,1.5) edge[-latex] (7.75,0.5);
    \draw (6.25,2.5) edge[-latex] (7.75,2.5);
	\draw (9,0.5) edge[-latex] (10.5,0.5);
	\draw (9,1.5) edge[-latex] (10.5,0.5);
    \draw (9,2.5) edge[-latex] (10.5,1.5);
	\draw (11.75,0.5) edge[-latex] (13.25,0.5);
	\draw (11.75,1.5) edge[-latex] (13.25,2.5);
    \draw (11.75,2.5) edge[-latex] (13.25,2.5);
    \draw (14.5,0.5) edge[-latex] (16,0.5);
	\draw (14.5,1.5) edge[-latex] (16,0.5);
    \draw (14.5,2.5) edge[-latex] (16,0.5);
\foreach \x in {0.5,2.5,3.25,5.25,6,8,8.75,10.75,11.5,13.5,14.25,16.25} \node at (\x,0.5) {\tiny 0};
\foreach \x in {0.5,2.5,3.25,5.25,6,8,8.75,10.75,11.5,13.5,14.25,16.25} \node at (\x,1.5) {\tiny 1};
\foreach \x in {0.5,2.5,3.25,5.25,6,8,8.75,10.75,11.5,13.5,14.25,16.25} \node at (\x,2.5) {\tiny 2};
\node at (1.5,-0.15)   {$1$};
\node at (4.25,-0.15)  {$ea$};
\node at (7,-0.15)     {$ae$};
\node at (9.75,-0.15)  {$a$};
\node at (12.5,-0.15)  {$e$};
\node at (15.25,-0.15) {$0$};
\end{tikzpicture}
\caption{Representing the elements of ${A}^1_2$ by the endomorphisms of~the chain $0<1<2$ that fix 0}\label{fig:A2asOPT-2}
\end{figure}

We denote by $\overline{\mA_2^1}$ the \ais{} obtained by equipping the monoid $A_2^1$ with addition induced by the order in Fig.~\ref{fig:a21-2}. This \ais{} is different from $\mA_2^1$, even though the two \ais{}s have isomorphic multiplicative and additive semigroups. However, the isomorphism between the multiplicative semigroups (which is simply the identity map) is incompatible with the isomorphism between the additive semigroups.

\subsection{Semiring identities and varieties}
\label{subsec:varieties} Fix a countably infinite set $X:=\{x,x_i,x_{ij},\dots\}$. The elements of $X$ are called \emph{variables}. A \emph{word} is a finite sequence of variables. (\emph{Semiring}) \emph{polynomials} are defined inductively starting from words: every word is a polynomial, and if $p_1$ and $p_2$ are polynomials, then so are the expressions $p_1+p_2$ and $p_1p_2$. As usual, when forming polynomials, multiplication takes precedence over addition.

Any map $\varphi\colon X\to S$, where $(S,+,\cdot)$ is an \ais{} is called a \emph{substitution}. The \emph{value} $p\varphi$ of a polynomial $p$ under $\varphi$ is defined inductively as follows: if $p=x_1\cdots x_k$ is a word, where $x_1,\dots,x_k$ are variables, then $p\varphi:=x_1\varphi\cdots x_k\varphi$; if $p=p_1+p_2$ or $p=p_1p_2$ for some polynomials $p_1$ and $p_2$ whose values under $\varphi$ have already been defined, then $p\varphi:=p_1\varphi+p_2\varphi$ or $p\varphi:=p_1\varphi\cdot p_2\varphi$, respectively.

A (\emph{semiring}) \emph{identity} is a formal equality between two polynomials, that is, an expression of the form $p=p'$, where $p$ and $p'$ are polynomials. An \ais{} $\mS=(S,+,\cdot)$ \emph{satisfies} $p=p'$ (or $p=p'$ \emph{holds} in $\mS$) if $p\varphi=p'\varphi$ for every substitution $\varphi\colon X\to S$. That is, each substitution of elements from $S$ for the variables occurring in $p$ or $p'$ yields equal values for the polynomials.

In view of the equivalence \eqref{eq:semilattice order}, we write identities of the form $p+p'=p'$ as $p\leqslant p'$, and we take the liberty of referring to such formal inequalities between two polynomials also as identities. For a simple concrete example, the reader can verify that the \ais{} $\overline{\mA_2^1}$ satisfies the identity $x^2\leqslant x$.

The class of all \ais{}s that satisfy all identities from a given set $\Sigma$ is called the \emph{variety defined by $\Sigma$}. The least variety containing an \ais{} $\mS$ is said to be \emph{generated by $\mS$}  and is denoted by $\var\mS$. In terms of identities, the variety $\var\mS$ is defined by all identities that hold in $\mS$. By the HSP Theorem (see \cite[Theorem II.9.5]{BuSa81}), $\var\mS$ coincides with the class of homomorphic images of subsemirings of direct powers of $\mS$. For instance, by \cite[Proposition 3]{GusVol25}, the \ais\ $\mB_2^1$ is a homomorphic image of a subsemiring of the direct square $\mA_2^1\times\mA_2^1$ whence $\mB_2^1$ belongs to the variety $\var\mA_2^1$. On the other hand, $\mB_2^1$ does not belong to the variety $\var\overline{\mA_2^1}$ since the identity $x^2\leqslant x$, which holds in $\overline{\mA_2^1}$, is easily seen to fail in $\mB_2^1$.

\section{The finite basis problem for \ais{}s: three methods}
\label{sec:methods}

An \ais{} variety is called \emph{finitely based} if it can be defined by a finite set of identities; otherwise, it is \emph{nonfinitely based}. An \ais{} $\mS$ is said to be \emph{finitely based} or \emph{nonfinitely based} according to whether the variety $\var\mS$ is finitely based or not. The finite basis problem for a class of \ais{}s is the issue of determining which \ais{}s in this class are finitely based and which are not.

For several classes of \ais{}s, the finite basis problem was studied in the first decade of the 2000s; see \cite{AEI03,AM11} and a series of papers by the first-named author~\cite{Dolinka07,Dolinka09a,Dolinka09b,Dolinka09c} for the main results of that period. A new activity phase has emerged since the beginning of the 2020s when new powerful approaches have been developed; see \cite{DolinkaGusVol24,Vol21,JackRenZhao22,GusVol23a,GusVol23b,RJZL23,WuRenZhao,WuZhaoRen23}.

As mentioned in the introduction, our solution to the finite basis problem for the endomorphism semirings of finite semilattices combines three methods that were developed to address this problem for various species of \ais{}s. These methods are discussed in Sections~\ref{subsec:infb}--\ref{subsec:kadourek}, but only to the extent necessary for their applications in our proofs.

\subsection{Inherently \nfb{} \ais{}s}
\label{subsec:infb}

An \ais{} variety is said to be \emph{locally finite} if each of its finitely generated members is finite. A locally finite variety is called \emph{inherently \nfb} if it is not contained in any finitely based locally finite variety. An \ais{} $\mS$ is \emph{inherently \nfb} whenever the variety $\var\mS$ it generates is inherently \nfb.

A variety generated by a finite \ais{} is called \emph{finitely generated}. Every finitely generated variety is locally finite (see \cite[Theorem II.10.16]{BuSa81}). Hence, a finite \ais{} $\mS$ is \nfb{} whenever the variety $\var\mS$ contains an inherently \nfb\ \ais. This yields an easily applicable tool for establishing nonfinite basability---provided an ``initial supply'' of inherently \nfb\ \ais{}s is available.

The first examples of inherently \nfb\ \ais{}s were found by the first-named author~\cite{Dolinka09a,Dolinka09b}, based on a sufficient condition for inherent nonfinite basability~\cite[Theorem B]{Dolinka09a}. In~\cite{Dolinka09a}, semirings were assumed to possess a neutral element for addition and were treated  as algebras with \emph{three} operations---binary addition and multiplication, and a nullary operation that returns this neutral element. In \cite{Dolinka09b}, it was additionally assumed that the neutral element for addition is also absorbing for multiplication. While these assumptions were essential for some results in~\cite{Dolinka09a,Dolinka09b}, a step-by-step analysis of the proof of \cite[Theorem B]{Dolinka09a} shows that the same proof remains valid in the present paper's setting, that is, for semirings whose operations are addition and multiplication only. Hence, the same sufficient condition of inherent nonfinite basability can be applied to such semirings.

We now introduce a few definitions in order to state the result we will require.

The sequence $\{Z_m\}_{m=1,2,\dots}$ of \emph{Zimin words} is defined inductively by
\[
Z_1:=x_1,\quad Z_{m+1}:=Z_mx_{m+1}Z_m,
\]
where $x_1,x_2,\dots,x_m,\dots$ are distinct variables. A word ${w}$ is \emph{minimal} [respectively, \emph{maximal}] for an \ais{} $\mS$ if the only word ${w'}$ such that $\mS$ satisfies the identity ${w'}\leqslant{w}$ [respectively, ${w}\leqslant{w'}$] is the word ${w}$ itself. A word $w$ is \emph{isolated} for $\mS$ if it is both minimal and maximal for $\mS$.

\begin{proposition}[{\!\cite[Theorem B]{Dolinka09a}}]
\label{prop:INFB}
If an \ais{} $\mS$ generates a locally finite variety and all Zimin words are isolated for $\mS$, then $\mS$ is inherently \nfb.
\end{proposition}

\subsection{Strongly \nfb{} \ais{}s}
\label{subsec:snfb}

A finite \ais{} is said to be \emph{strongly \nfb} if it does not belong to any \ais{} variety that is both finitely generated and finitely based. Clearly, every strongly \nfb{} \ais{} is \nfb.

Since every finitely generated variety is locally finite, every inherently \nfb{} finite \ais{} is strongly \nfb; however, there exist strongly \nfb{} \ais{}s that are not inherently \nfb; see~\cite[Example 6.7]{JackRenZhao22}. Thus, strong nonfinite basability offers a tool for establishing the nonfinite basability of finite \ais{}s whose scope exceeds that of inherent nonfinite basability.

We will use the following sufficient condition for strong nonfinite basability.

\begin{proposition}[{\!\cite[Theorem 6.1]{JackRenZhao22}}]
\label{prop:SNFB}
Let $\mS=(S,+,\cdot)$ be a finite \ais. If the semigroup $(S,\cdot)$ contains a nonabelian nilpotent subgroup, then $\mS$ is strongly \nfb.
\end{proposition}

\subsection{AI-semirings whose varieties contain $\mB_2^1$}
\label{subsec:kadourek}

The six-element \ais{} $\mB_2^1$ introduced in Section~\ref{subsec:b1a1} was shown to be \nfb{} in \cite{JackRenZhao22,Vol21}. It remains unknown whether $\mB_2^1$ possesses either of the more powerful properties discussed in Sections~\ref{subsec:infb} and~\ref{subsec:snfb}, i.e., inherent or strong nonfinite basability. However, the nonfinite basability of this \ais{} has proved to be quite contagious: if the variety generated by an \ais{} $\mS$ contains $\mB_2^1$, then, under rather mild conditions, $\mS$ is also \nfb. See \cite{GusVol23a,GusVol23b,GusVol25} for various instances of such conditions; in the present paper, we adopt the one from \cite{GusVol23b}. This condition is formulated in terms of certain identities, which we present for the sake of completeness, even though their precise form plays absolutely no role in our arguments.

For any $k,h\ge1$, consider the set of variables
\[
X_k^{(h)}:=\{x_{i_1i_2\cdots i_h}\mid 1\le i_1,i_2,\dots,i_h\le k\}.
\]
For arbitrarily fixed $n,m\ge1$, we introduce a family of words $v_{n,m}^{(h)}$ over $X_{2n}^{(h)}$, $h=1,2,\dotsc$, by induction on $h$.

For $h=1$, let
\[
v_{n,m}^{(1)}:=x_1x_2\cdots x_{2n}\,\Bigl(x_nx_{n-1}\cdots x_1\cdot x_{n+1}x_{n+2}\cdots x_{2n}\Bigr)^{2m-1}.
\]
Now assume $h>1$ and that the word $v_{n,m}^{(h-1)}$ over $X_{2n}^{(h-1)}$ has already been defined. For each $j\in\{1,2,\dots,2n\}$, let $v_{n,m,j}^{(h-1)}$ be the word over $X_{2n}^{(h)}$ obtained from  $v_{n,m}^{(h-1)}$ by appending $j$ to the indices of the variables in $X_{2n}^{(h-1)}$; that is, each  variable $x_{i_1i_2\dots i_{h-1}}$ is replaced by $x_{i_1i_2\dots i_{h-1}j}$ for all $i_1,i_2,\dots,i_{h-1}\in\{1,2,\dots,2n\}$. Then let
\[
v_{n,m}^{(h)}:=v_{n,m,1}^{(h-1)}\cdots v_{n,m,2n}^{(h-1)}\Bigl(v_{n,m,n}^{(h-1)}\cdots v_{n,m,1}^{(h-1)}\cdot v_{n,m,n+1}^{(h-1)}\cdots v_{n,m,2n}^{(h-1)}\Bigr)^{2m-1}.
\]

\begin{proposition}[{\!\cite[Theorem 4.2]{GusVol23b}}]
\label{prop:B21}
If an \ais{} $\mS$ satisfies the identities $v_{n,m}^{(h)}=(v_{n,m}^{(h)})^2$ for all $n\ge2$ and some $m,h\ge 1$ and the variety $\var\mS$ contains the \ais{} $\mB_2^1$, then $\mS$ is \nfb.
\end{proposition}

Two particular \ais{}s to which this result was applied in \cite{GusVol23b} play a role in the present paper. A \emph{rook} matrix is a zero-one square matrix with at most one entry equal to 1 in each row and each column. (The name refers to the fact that $8\times 8$ rook matrices encode placements of nonattacking rooks on a chessboard.) Denote by $R_t$ the set of all $t\times t$ rook matrices.  The \emph{rook semiring} $\mathcal{R}_t:=(R_t,+,\cdot)$ arises by defining $+$ as the entrywise (Hadamard) product of matrices:
\[
a+b:=(a_{ij}b_{ij})_{t\times t}\ \text{ for all }\ a = (a_{ij})_{t \times t}\ \text{ and }\ b = (b_{ij})_{t \times t}\ \text{ in }\ R_t,
\]
while $\cdot$ denotes the usual matrix multiplication.

\begin{proposition}[{\!\cite[Proposition 2.5]{GusVol23b}}]
\label{prop:rook2and3}
For any $n\ge 2$, the rook semirings $\mathcal{R}_2$ and $\mathcal{R}_3$ satisfy the identities $v_{n,2}^{(2)}=(v_{n,2}^{(2)})^2$ and $v_{n,6}^{(4)}=(v_{n,6}^{(4)})^2$, respectively.
\end{proposition}

\section{Main result and its proof}
\label{sec:end}

Our main result classifies the endomorphism semirings of finite semilattices with respect to the finite basis problem. Its formulation involves two parameters of a finite semilattice $(A,+)$, namely, \emph{size} (the number of elements in $A$) and \emph{height} (the maximal size of chains in the partially ordered set $(A,\leqslant)$, where the order $\leqslant$ is defined via \eqref{eq:semilattice order}).

\begin{theorem}
\label{thm:end}
The endomorphism semiring $\End(\mA)$ of a finite semilattice $\mA$ is \fb{} if and only if the size of $\mA$ is 1 or 2.

Moreover, the semiring $\End(\mA)$ is inherently \nfb{} if the height of $\mA$ is at least 3, and is strongly \nfb{} if the size of $\mA$ is at least 5.
\end{theorem}

The statement of Theorem~\ref{thm:end} consists of four claims, which we now isolate and then prove one by one.

\begin{claim}
\label{claim:if}
The endomorphism semiring of a semilattice of size 1 or 2 is \fb.
\end{claim}

\begin{claim}
\label{claim:infb}
The endomorphism semiring of a finite semilattice of height at least 3 is inherently \nfb.
\end{claim}

\begin{claim}
\label{claim:snfb}
The endomorphism semiring of a finite semilattice of size at least 5 is strongly \nfb.
\end{claim}

\begin{claim}
\label{claim:onlyif}
The endomorphism semiring of a finite semilattice of size at least 3 is \nfb.
\end{claim}

\begin{proof}[Proof of Claim~\ref{claim:if}] The endomorphism semiring of the one-element semilattice is the one-element semiring, which is trivially \fb.

Up to isomorphism, the only two-element semilattice is the chain $\mC_2:=(\{0,1\},+)$ with addition derived from the order $0<1$. It has three endomorphisms: the two constant maps \ $\begin{matrix}0\mapsto 0\\ 1\mapsto 0\end{matrix}$ \ and \ $\begin{matrix}0\mapsto 1\\ 1\mapsto 1\end{matrix}$\,, and the identity map \ $\begin{matrix}0\mapsto 0\\ 1\mapsto 1\end{matrix}$\,. Each of these three maps is idempotent, that is, equal to its own square. Hence, multiplication in the endomorphism semiring  $\End(\mC_2)$ is idempotent. It is known \cite[Theorem 4.1]{Pastijn05} that every \ais{} with idempotent multiplication is \fb.
\end{proof}

For Claim~\ref{claim:infb}, we require two properties of the semirings $\overline{\mA_2^1}$ and $\mB_2^1$ from Section~\ref{subsec:b1a1}.

\begin{lemma}
\label{lem:max}
All Zimin words are maximal for the \ais{} $\overline{\mA_2^1}$.
\end{lemma}

\begin{proof}
This is established in the proof of~\cite[Theorem~10]{Dolinka09b}, where the \ais{} $\overline{\mA_2^1}$ appears under notation $\mathbf{O}_2$.
\end{proof}

\begin{lemma}
\label{lem:min}
All Zimin words are minimal for the \ais{} $\mB_2^1$.
\end{lemma}

\begin{remark}
The reviewer has informed us that the result of Lemma~\ref{lem:min} already appeared in the Master's thesis \cite{Yi24}. As the thesis is not readily accessible, we have retained the proof of the lemma for the reader’s convenience.
\end{remark}

\begin{proof}
We have to show that for every $m=1,2,\dots$, and every identity $w\leqslant Z_m$ that holds in $\mB_2^1$, the word $w$ coincides with the Zimin word $Z_m$. We proceed by induction on $m$.

When $m=1$, we have $Z_1=x_1$. Define a substitution $\varphi\colon X\to B_2^1$ by setting $x_1\varphi=c$ and $x\varphi=0$ for all variables $x\ne x_1$. Since $w\leqslant x_1$ holds in $\mB_2^1$, we must have $w\varphi\leqslant c=x_1\varphi$ in the semilattice $(B_2^1,\leqslant)$, whence $w\varphi=c$ because $c$ is a minimal element of the semilattice (see Fig.~\ref{fig:b21}). By the definition of $\varphi$ and the fact that $c^2=0$ in $\mB_2^1$, the equality $w\varphi=c$ is only possible if no variable other than $x_1$ occurs in $w$ and $x_1$ occurs in $w$ exactly once. Hence $w$ coincides with $Z_1=x_1$.

Now let $m>1$. If $w=x_1^k$ for some positive integer $k$, consider a substitution $\chi\colon X\to B_2^1$ such that $x_1\chi=1$, $x_2\chi=c$ and $x\chi=d$ for all $x\ne x_1,x_2$. Then $w\chi=1\nleqslant Z_m\chi=c$, which contradicts the assumption that the identity $w\leqslant Z_m$ that holds in $\mB_2^1$. Thus, $w$ must contain a variable other than $x_1$, and hence the word $w'$, obtained from $w$ by removing all occurrences of $x_1$, is nonempty.

Substituting the element 1 for the variable $x_1$ in the identity $w\leqslant Z_m$, we get that $\mB_2^1$ satisfies the identity $w'\leqslant Z'_m$ where the word $Z'_m$ is obtained from $Z_m$ by removing all occurrences of $x_1$. The word $Z'_m$ is nothing but the Zimin word $Z_{m-1}$ with the subscript of each variable increased by 1. By the induction assumption, the word $w'$ coincides with $Z'_m$. Since the word $w$ becomes $Z'_m$ after removing all occurrences of $x_1$, we conclude that
\[
{w}=x_1^{\varepsilon_1}x_2x_1^{\varepsilon_2}x_3x_1^{\varepsilon_3}x_2x_1^{\varepsilon_4}x_4\cdots x_3x_1^{\varepsilon_{2^{m-1}-1}}x_2x_1^{\varepsilon_{2^{m-1}}},
\]
where $\varepsilon_i$, $i=1,2,\dots,2^{m-1}$, are nonnegative integers. (If $\varepsilon_i=0$, we understand $x_1^{\varepsilon_i}$ as the empty word.) Thus, the word $w$ is obtained from $Z_m$ by substituting the word $x_1^{\varepsilon_i}$ for the $i$-th from the left occurrence of the variable $x_1$.

Define a substitution $\psi\colon X\to B_2^1$ by setting $x_1\psi=c$ and $x\psi=d$ for all variables $x\ne x_1$. Then $Z_m\psi=c$ due to the equality $cdc=c$ in $\mB_2^1$. Since $w\leqslant Z_m$ holds in $\mB_2^1$, we must have $w\psi\leqslant c=Z_m\psi$ in $(B_2^1,\leqslant)$, whence $w\psi=c$ because of minimality of $c$.

If $\varepsilon_i>1$ for some $i=1,2,\dots,2^{m-1}$, then the word $x_1^2$ occurs in ${w}$. Then $w\psi$ factors through $x_1^2\psi=c^2=0$ so $w\psi=0\ne c$.

If $\varepsilon_i=0$ for some $i=2,\dots,2^{m-1}-1$, then for some $x_j$ and $x_k$ with $j,k>1$, the word $x_jx_k$ occurs in ${w}$. Then $w\psi$ factors through $(x_jx_k)\psi=d^2=0$ so $w\psi=0\ne c$.

Thus, $\varepsilon_i=1$ for all $i=2,\dots,2^{m-1}-1$ and $\varepsilon_1,\varepsilon_{2^{m-1}}\in\{0,1\}$. If $\varepsilon_1=0$, then $w\psi\in\{d,dc\}$ due to the equality $dcd=d$ in $\mB_2^1$. Again, $w\psi\ne c$. Similarly, if $\varepsilon_{2^{m-1}}=0$, then $w\psi\in\{cd,d\}$ and $w\psi\ne c$.

We have shown that the equality $w\psi=c$ forces $\varepsilon_i=1$ for each $i=1,2,\dots,2^{m-1}$. Hence, the word ${w}$ coincides with $Z_m$.
\end{proof}

To complete the proof of Claim~\ref{claim:infb}, we require an intermediate result of independent interest. By the $n$-element chain, we mean the semilattice $\mC_n:=(\{0,1,\dots,n-1\},+)$, whose addition is induced by the natural order $0<1<\cdots<n-1$. The finite basis problem for the endomorphism semirings $\End(\mC_n)$ was addressed in~\cite{GusVol25}, where it was proved that the \ais{} $\End(\mC_3)$ is \nfb{} \cite[Theorem 2]{GusVol25}. Here, we establish the following much stronger fact via a much simpler proof.

\begin{proposition}
\label{prop:O3}
The \ais{} $\End(\mC_3)$ is inherently \nfb.
\end{proposition}

\begin{proof}
It is immediate from the definitions that if a word $w$ is {minimal} [{maximal}] for an \ais{} $\mS$, then $w$  is {minimal} [respectively, {maximal}] for any \ais{} $\mT$ such that the variety $\var\mT$ contains $\mS$. In Section~\ref{subsec:b1a1}, it is shown that the \ais{} $\overline{\mA_2^1}$ is isomorphic to a subsemiring of $\End(\mC_3)$; see Fig.~\ref{fig:A2asOPT-2}. Hence, we have $\overline{\mA_2^1}\in\var\End(\mC_3)$, and using Lemma~\ref{lem:max}, we conclude that all Zimin words are maximal for the \ais{} $\End(\mC_3)$.

Similarly, the \ais{} $\mA_2^1$ is isomorphic to a subsemiring of $\End(\mC_3)$; see Fig.~\ref{fig:A2asOPT}. Hence, the variety $\var\End(\mC_3)$ contains $\mA_2^1$. As discussed at the end of Section~\ref{subsec:varieties},
the variety $\var\mA_2^1$ contains $\mB_2^1$, whence $\mB_2^1\in\var\End(\mC_3)$. Using Lemma~\ref{lem:min}, we conclude that all Zimin words are minimal for the \ais{} $\End(\mC_3)$. Thus, all Zimin words are isolated for  $\End(\mC_3)$, and invoking Proposition~\ref{prop:INFB} completes the proof.
\end{proof}

\begin{proof}[Proof of Claim~\ref{claim:infb}] We aim to show that if a finite semilattice $\mA=(A,+)$ is such that the partially ordered set $(A,\leqslant)$ contains a three-element chain, then the endomorphism semiring of $\mA$ is inherently \nfb.

Take a chain $v_0<v_1<v_2$ in $(A,\leqslant)$ such that $v_2$ is the greatest element of $(A,\leqslant)$ and define the subsets
\begin{align*}
A_0&:=\{a\in A\mid a\leqslant v_0\},\\
A_1&:=\{a\in A\mid a\leqslant v_1,\ a\nleqslant v_0\},\\
A_2&:=\{a\in A\mid a\leqslant v_2,\ a\nleqslant v_1\}.
\end{align*}
By the choice of $v_2$, the sets $A_0,A_1,A_2$ form a partition of $A$ illustrated in Fig.~\ref{fig:partition}.
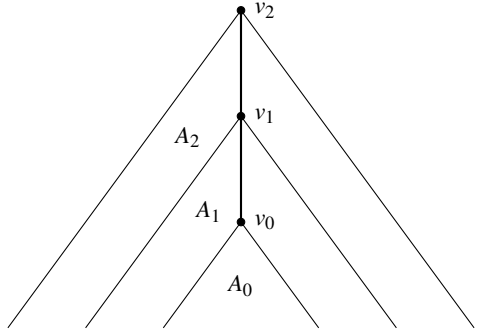
\begin{figure}[ht]
\begin{tikzpicture}[scale=1.4, every node/.style={font=\small}]

\coordinate (a1) at (0,0);
\coordinate (a2) at (0,1);
\coordinate (a3) at (0,2);

\def\w{2.2}

\begin{scope}
  \draw ($(a1)+(-\w/3, -1)$) -- (a1) -- ($(a1)+(\w/3, -1)$);
  \node at ($(a1)+(0,-.6)$) {$A_0$};
\end{scope}

\begin{scope}
 \draw ($(a2)+(-\w/1.5, -2)$) -- (a2) -- ($(a2)+(\w/1.5, -2)$);
  \node at ($(a2)+(-.3,-.9)$) {$A_1$};
\end{scope}

\begin{scope}
 \draw ($(a3)+(-\w, -3)$) -- (a3) -- ($(a3)+(\w, -3)$);
  \node at ($(a3)+(-.5,-1.2)$) {$A_2$};
\end{scope}

\node[draw, circle, fill=black, inner sep=1pt, label=right:$v_0$] at (a1) {};
\node[draw, circle, fill=black, inner sep=1pt, label=right:$v_1$] at (a2) {};
\node[draw, circle, fill=black, inner sep=1pt, label=right:$v_2$] at (a3) {};

\draw[thick] (a1) -- (a2) -- (a3);

\end{tikzpicture}
\caption{The partition of $A$ induced by the  chain $v_0<v_1<v_2$}\label{fig:partition}
\end{figure}

It is easy to see that the map $\pi\colon A\to\{0,1,2\}$ defined by setting $a\pi=i$ if and only if $a\in A_i$ is a homomorphism of the semilattice $\mA$ onto the three-element chain $\mC_3$. Now, for every endomorphism $\gamma\in\End(\mC_3)$, define a function $\overline{\gamma}\colon A\to A$ by setting $a\overline{\gamma}=v_{a\pi\gamma}$. In the visual terms of Fig.~\ref{fig:partition}, the elements of $A_i$ are sent by $\overline{\gamma}$ \ to the vertex $v_{i\gamma}$ of the `cone' $A_{i\gamma}$.

For any $a,b\in A$, we have
\begin{align*}
 (a+b)\overline{\gamma}&=v_{(a+b)\pi\gamma}&&\text{by definition} \\
                       &=v_{(a\pi+b\pi)\gamma} &&\text{since $\pi$ is a semilattice homomorphism}\\
                       &=v_{a\pi\gamma+b\pi\gamma}&&\text{since $\gamma$ is a semilattice endomorphism.}
\end{align*}
Semilattice addition in each of the chains $v_0<v_1<v_2$ and $0<1<2$ yields the maximum of the two summands. Therefore, for all $i,j\in\{0,1,2\}$,
\begin{equation}\label{eq:max}
v_i+v_j=\max\{v_i,v_j\}=v_{\max\{i,j\}}=v_{i+j}.
\end{equation}
Using \eqref{eq:max}, we obtain
\[
v_{a\pi\gamma+b\pi\gamma}=v_{\max\{a\pi\gamma,\,b\pi\gamma\}}=\max\{v_{a\pi\gamma},\,v_{b\pi\gamma}\}=v_{a\pi\gamma}+v_{b\pi\gamma}=a\overline{\gamma}+b\overline{\gamma}.
\]
We have thus verified that $\overline{\gamma}$ is an endomorphism of the semilattice $\mA$.

By the definition of the map $\pi$, we have $v_i\pi=i$ for each $i=0,1,2$. Hence,
\begin{equation}\label{eq:compose}
v_i\overline{\gamma}=v_{i\gamma}
\end{equation}
for every endomorphism $\gamma\in\End(\mC_3)$, so $\gamma$ can be recovered from the restriction of $\overline{\gamma}$ to the set $\{v_0,v_1,v_2\}$. This shows that the map $\gamma\mapsto\overline{\gamma}$ is one-to-one.

Let now $\gamma_1,\gamma_2\in\End(\mC_3)$ and let $a\in A$ be an arbitrary element. We have
\begin{gather*}
a(\overline{\gamma_1+\gamma_2})=v_{a\pi(\gamma_1+\gamma_2)}=v_{a\pi\gamma_1+a\pi\gamma_2}\stackrel{\eqref{eq:max}}=v_{a\pi\gamma_1}+v_{a\pi\gamma_2}=a\overline{\gamma_1}+a\overline{\gamma_2},\\
a(\overline{\gamma_1\gamma_2})=v_{a\pi(\gamma_1\gamma_2)}=v_{(a\pi\gamma_1)\gamma_2}\stackrel{\eqref{eq:compose}}=v_{a\pi\gamma_1}\overline{\gamma_2}=(a\overline{\gamma_1})\overline{\gamma_2}= a(\overline{\gamma_1}\cdot\overline{\gamma_2}).
\end{gather*}
Since $a\in A$ was arbitrary, we conclude that $\overline{\gamma_1+\gamma_2}=\overline{\gamma_1}+\overline{\gamma_2}$ and $\overline{\gamma_1\gamma_2}=\overline{\gamma_1}\cdot\overline{\gamma_2}$. This means that
the map $\gamma\mapsto \overline{\gamma}$ is a semiring homomorphism.

We have thus shown that the \ais{} $\End(\mC_3)$ is isomorphic to a subsemiring of the \ais{} $\End(\mA)$. Now Proposition~\ref{prop:O3} applies, yielding that $\End(\mA)$ is inherently \nfb.
\end{proof}

Since Claim~\ref{claim:infb} has been established, we only need to consider finite semilattices of height 2 in the subsequent proofs of Claims~\ref{claim:snfb} and~\ref{claim:onlyif}. For each $t\ge 2$, there is a unique, up to isomorphism, semilattice of height 2 and size $t+1$, namely, the $t$-\emph{pod} $\mP_t$ shown in Fig.~\ref{fig:Ln}.
\begin{figure}[ht]
\begin{tikzpicture}[x=1cm,y=1.2cm]
\node at (0,2)    (0)  {$0$};
\node at (-3,1)  (1) {$1$};
\node at (-1.5,1)  (2) {$2$};
\node at (0,1)    (3)  {$3$};
\node at (1.5,1)  (cdots) {$\cdots$};
\node at (3,1)  (n) {$t$};
\draw (0)   -- (1);
\draw (0)  --  (2);
\draw (0)  --  (3);
\draw (0)  --  (n);
\end{tikzpicture}
\caption{The semilattice $\mP_t$}\label{fig:Ln}
\end{figure}
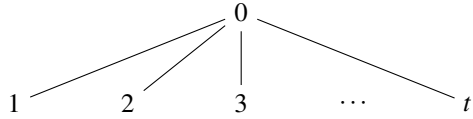

The properties of the endomorphism semiring of the $t$-pod relevant for the proofs of Claims~\ref{claim:snfb} and~\ref{claim:onlyif} are collected in the following proposition.

\begin{proposition}
\label{prop:n-pod}
The set $\End^0(\mP_t)$ of all endomorphisms of the $t$-pod $\mP_t$ that fix its greatest element $0$ forms a subsemiring of $\End(\mP_t)$. This subsemiring is isomorphic to the rook semiring $\mR_t$ and all endomorphisms in $\End(\mP_t)\setminus\End^0(\mP_t)$ are constant maps.
\end{proposition}

\begin{proof}
That $\End^0(\mP_t)$ is a subsemiring of $\End(\mP_t)$ is immediate.

For any $\alpha\in\End^0(\mP_t)$, let $M^\alpha$ denote the $t\times t$ matrix whose rows and columns are numbered by $1,2,\dots,t$, and whose entry in position $(i,j)$ is 1 if $i\alpha=j$ and 0 otherwise. By definition, each row of $M^\alpha$ contains at most one entry equal to 1. Suppose that some column of $M^\alpha$ contains 1s in two different rows $i$ and $i'$. Then $i\alpha=i'\alpha$, whence
\[
(i+i')\alpha=i\alpha+i'\alpha=i\alpha+i\alpha=i\alpha
\]
since $\alpha$ is an endomorphism of $\mP_t$ and addition in $\mP_t$ is idempotent. However, $i+i'=0$ since $i\ne i'$, and so $(i+i')\alpha=0$ because $\alpha$ fixes 0. This contradicts $i\alpha\ne 0$. Thus, each column of $M^\alpha$ also contains at most one entry equal to 1, and hence $M^\alpha$ is a rook matrix.

Conversely, for any $t\times t$ rook matrix $M=(m_{ij})$, let $\alpha_M$ denote the self-map on the set $\{0,1,\dots,t\}$ defined by setting $0\alpha_M:=0$ and
\[
i\alpha_M:=\begin{cases}
j&\text{if }\ m_{ij}=1,\\
0&\text{otherwise},
\end{cases}\ \text{ for each }\ i=1,2,\dots,t.
\]
The map is well defined as each row of $M$ contains at most one 1. Since each column of $M$ also contains at most one 1, for any two different $i,i'\in\{1,\dots,t\}$, we have either $i\alpha_M\ne i'\alpha_M$ or $i\alpha_M= i'\alpha_M=0$. In either case, $(i+i')\alpha_M=i\alpha_M+i'\alpha_M$ as both sides are equal to 0. Hence, $\alpha_M$ is an endomorphism of $\mP_t$ that fixes 0.

By construction, $M^{\alpha_M}=M$ and $\alpha_{M^\alpha}=\alpha$. Thus, the maps $\alpha\mapsto M^\alpha$ and $M\mapsto\alpha_M$ form mutually inverse bijections between $\End^0(\mP_t)$ and the set $R_t$ of all $t\times t$ rook matrices. It is routine to verify that these maps preserve both + and $\cdot$. Hence, the subsemiring $\End^0(\mP_t)$ and the rook semiring $\mR_t$ are isomorphic.

Finally, any $\gamma\in\End(\mP_t)\setminus\End^0(\mP_t)$ does not fix 0. For any $i\in\{0,1,\dots,t\}$, we have $i\leqslant 0$ whence $i\gamma\leqslant 0\gamma$ as the endomorphism $\gamma$ preserves the order $\leqslant$. Since $0\gamma\ne 0$, the element $0\gamma$ is minimal with respect to  $\leqslant$, and therefore, the inequality $i\gamma\leqslant 0\gamma$ implies the equality $i\gamma=0\gamma$. Thus, $\gamma$ is the constant map that sends each $i\in\{0,1,\dots,t\}$ to $0\gamma$.
\end{proof}

\begin{proof}[Proof of Claim~\ref{claim:snfb}] We aim to show that the endomorphism semiring of a finite semilattice containing at least five elements is strongly \nfb. Claim~\ref{claim:infb}, together with the fact that inherently \nfb{} semirings are strongly \nfb, allows us to restrict to the case where the semilattice has height 2. In that case, it is isomorphic to the $t$-pod $\mP_t$ with $t\ge 4$.

By Proposition~\ref{prop:n-pod}, the endomorphism semiring $\End(\mP_t)$ contains a subsemiring isomorphic to the rook semiring $\mR_t$. Recall that a \emph{permutation} matrix is a zero-one square matrix with exactly one entry equal to 1 in each row and each column. The $t\times t$ permutation matrices form a subgroup in the multiplicative semigroup of $\mR_t$, and this subgroup is isomorphic to the group of all permutations of the set $\{1,2,\dots,t\}$. For $t\ge 4$, the latter group possesses non-abelian nilpotent subgroups, for instance, the subgroup generated by permutations (1234) and (13), which is isomorphic to the eight-element dihedral group (i.e., the symmetry group of a square).

Proposition~\ref{prop:SNFB} now applies, showing that the semiring $\End(\mP_t)$ with  $t\ge 4$ is strongly \nfb.
\end{proof}

\begin{proof}[Proof of Claim~\ref{claim:onlyif}] We aim to show that the endomorphism semiring of a finite semilattice containing at least three elements is \nfb. Claims~\ref{claim:infb} and~\ref{claim:snfb} reduce the task to the case where the semilattice has height 2 and size 3 or 4. In that case, it is isomorphic to the $t$-pod $\mP_t$ with $t\in\{2,3\}$.

We will employ Proposition~\ref{prop:B21}, so we now show that the semirings $\End(\mP_2)$ and $\End(\mP_3)$ satisfy its conditions.

By Proposition~\ref{prop:n-pod}, the subsemiring $\End^0(\mP_t)$ of $\End(\mP_t)$ is isomorphic to the semiring of all $t\times t$ rook matrices. The $2\times 2$ rook matrices are listed in~\eqref{eq:r2}.
\begin{equation}
\label{eq:r2}
\begin{tabular}{ccccccc}
$\left(\begin{matrix} 0&0\\0&0\end{matrix}\right)$
&
$\left(\begin{matrix} 1&0\\0&0\end{matrix}\right)$
&
$\left(\begin{matrix} 0&1\\0&0\end{matrix}\right)$
&
$\left(\begin{matrix} 0&0\\1&0\end{matrix}\right)$
&
$\left(\begin{matrix} 0&0\\0&1\end{matrix}\right)$
&
$\left(\begin{matrix} 1&0\\0&1\end{matrix}\right)$
&
$\left(\begin{matrix} 0&1\\1&0\end{matrix}\right)$\\
$\downarrow$ & $\downarrow$ & $\downarrow$ & $\downarrow$ & $\downarrow$ & $\downarrow$ &\\
0 & $cd$ & $c$ & $d$ & $dc$ & 1 &
\end{tabular}
\end{equation}
It is well known (and easy to verify) that the first six matrices in \eqref{eq:r2} form a subsemiring of $\mR_2$, and this subsemiring is isomorphic to the \ais{} $\mB_2^1$ under the map indicated in~\eqref{eq:r2}. Hence, the variety $\var\End(\mP_2)$ contains  $\mB_2^1$. By padding the matrices from \eqref{eq:r2} with zeros  to form $3 \times 3$ matrices, we obtain an embedding of $\mathcal{R}_2$ into $\mathcal{R}_3$, whence the variety $\var\End(\mP_3)$ also contains $\mB_2^1$. This verifies one of the two conditions of Proposition~\ref{prop:B21}.

To verify the other condition of Proposition~\ref{prop:B21}, we first show that for every $t\ge 2$ and every word $w$, the identity $w=w^2$ holds in the endomorphism semiring $\End(\mP_t)$ whenever it holds in the rook semiring $\mR_t$. We need to verify that the words $w$ and $w^2$ evaluate to the same element under an arbitrary substitution $\varphi\colon X\to\End(\mP_t)$. If $x\varphi\in\End^0(\mP_t)$ for all variables $x$ occurring in $w$, then the evaluation takes place entirely within the subsemiring $\End^0(\mP_t)$, which is isomorphic to the rook semiring $\mR_t$ by Proposition~\ref{prop:n-pod}. Since the identity $w=w^2$ holds in $\mR_t$, we have $w\varphi=w^2\varphi$ in this case. Otherwise, let $x$ be a variable occurring in $w$ such that $x\varphi\notin\End^0(\mP_t)$, and fix a decomposition of the word $w$ as $w=w'xw''$, where $w'$ and $w''$ are some (possibly empty) words. As shown in Proposition~\ref{prop:n-pod}, $x\varphi$ is a constant map, and thus satisfies $\alpha\cdot x\varphi=x\varphi$ for all $\alpha\in\End(\mP_t)$. Using this, we compute
\begin{align*}
  w\varphi&=w'\varphi\cdot x\varphi\cdot w''\varphi=x\varphi\cdot w''\varphi, \\
  w^2\varphi&=(ww')\varphi\cdot x\varphi\cdot w''\varphi=x\varphi\cdot w''\varphi.
\end{align*}
(If $w'$ or $w''$ is empty, then $w'\varphi$, or respectively $w''\varphi$, is understood as the identity endomorphism.) Hence, $w\varphi=w^2\varphi$ in this case as well.

Applying the fact just established to the identities $v_{n,2}^{(2)} = (v_{n,2}^{(2)})^2$ and $v_{n,6}^{(4)} = (v_{n,6}^{(4)})^2$, which, by Proposition~\ref{prop:rook2and3}, hold in the rook semirings $\mathcal{R}_2$ and $\mathcal{R}_3$, respectively, for all $n \ge 2$, we conclude that these identities also hold in the semirings $\End(\mP_2)$ and $\End(\mP_3)$, respectively. Proposition~\ref{prop:B21} now applies, showing that both semirings are \nfb.
\end{proof}

The proof of Theorem~\ref{thm:end} is now complete.

\section{Future work}
\label{sec:conclude}

Theorem~\ref{thm:end} provides a complete classification of the endomorphism semirings of finite semilattices with respect to the property of being finitely based or \nfb. Its proof also involves two stronger notions of nonfinite basability: inherent and strong nonfinite basability. From Claims~\ref{claim:if} and~\ref{claim:infb} in Section~\ref{sec:end}, one can immediately extract a complete classification of the endomorphism semirings of finite \emph{chains} with respect to all three notions of nonfinite basability, thereby improving on the results of \cite{GusVol25}.
\begin{corollary}
\label{cor:chains} For the $n$-element chain $\mC_n$, the following are equivalent:
\begin{enumerate}
  \item[\emph{(i)}] the semiring $\End(\mC_n)$ is inherently \nfb;
  \item[\emph{(ii)}] the semiring $\End(\mC_n)$ is strongly \nfb;
  \item[\emph{(iii)}] the semiring $\End(\mC_n)$ is \nfb;
  \item[\emph{(iv)}] the inequality $n\ge 3$ holds.
\end{enumerate}
\end{corollary}

However, for the endomorphism semirings of arbitrary finite semilattices, we still have no definitive results regarding to inherent and strong nonfinite basability. The critical issue here is the following
\begin{question}[{\!\cite[Problem 7.7(2)]{JackRenZhao22}}]
\label{ques:b21} Is the \ais{} $\mB_2^1$ inherently \nfb?
\end{question}

If the answer to Question~\ref{ques:b21} is affirmative, then it follows readily from the proof of Theorem~\ref{thm:end} that the three notions of nonfinite basability are equivalent for the endomorphism semirings of finite semilattices. However, we rather expect a negative answer.

As for the two other six-element \ais{}s from Section~\ref{subsec:b1a1}, the semiring $\mA_2^1$ is known to be \nfb{}. This was established in~\cite[Theorem 2]{GusVol25} by a rather involved argument based on the approach outlined in Section~\ref{subsec:kadourek}. The question of whether $\mA_2^1$ is inherently or strongly \nfb{} remains open---and so does the following
\begin{question}[{\!\cite[Question 9]{Dolinka09b}\footnote{We have already mentioned that in \cite{Dolinka09b}, the \ais{} $\overline{\mA_2^1}$ appears as $\mathbf{O}_2$.}}]
\label{ques:a21} Is the \ais{} $\overline{\mA_2^1}$ \ \fb?
\end{question}
The monoid $A_2^1$, that is, the multiplicative semigroup of the semirings  $\mA_2^1$ and $\overline{\mA_2^1}$, is inherently \nfb. This makes Question~\ref{ques:a21} particularly intriguing, as no example is currently known of a \fb{} \ais{} whose multiplicative semigroup is inherently \nfb. For comparison, we note that there do exist \fb{} \ais{}s whose multiplicative semigroups are strongly \nfb.

\subsection*{Acknowledgement} The authors are grateful to the anonymous referee for their feedback and helpful suggestions.

\end{document}